\documentclass[floats,reqno,oneside]{amsart}
\usepackage{amsmath,amsfonts,amssymb,amsthm,stmaryrd,bm}
\usepackage[mathscr]{eucal}
\usepackage{mathtools}
\usepackage{wasysym}
\usepackage{graphicx}
\usepackage{hyperref}
\usepackage{breakurl}

\numberwithin{equation}{section}
\newcommand{\beq}{\begin{equation}}
\newcommand{\eeq}{\end{equation}}
\newcommand{\beqa}{\begin{eqnarray}}
\newcommand{\eeqa}{\end{eqnarray}}
\newcommand{\beqan}{\begin{eqnarray*}}
\newcommand{\eenan}{\end{eqnarray*}}

\def\Cal{\mathcal}

\newcommand{\Dslash}{{\slash{\kern -0.5em}\partial}}
\newcommand{\Aslash}{{\slash{\kern -0.5em}A}}

\def\sqr#1#2{{\vcenter{\hrule height.#2pt
     \hbox{\vrule width.#2pt height#1pt \kern#1pt
        \vrule width.#2pt}
     \hrule height.#2pt}}}

\def\thinspace{\kern .16667em}

\def\xp{x_{{\kern -.2em}_\perp}}
\def\subp{_{{\kern -.2em}_\perp}}

\makeatletter
\newcommand*{\defeq}{\mathrel{\rlap{%
                     \raisebox{0.3ex}{$\m@th\cdot$}}%
                     \raisebox{-0.3ex}{$\m@th\cdot$}}%
                     =}
\makeatother

\newcommand{\dom}{\mathrm{dom}\,}

\newcommand{\Nat}{{\mathbb N}}

\newcommand\Star{{{}^*}\!}

\newcommand{\Pow}{\Cal P}
\newcommand{\Powfin}{{\Cal P_\omega}}
\newcommand{\Powcof}{\overline{\Powfin}}
\newcommand{\Powtop}{{\Pow}_*}
\newcommand{\setbld}[2]{\{#1 \,|\, #2 \}}

\newcommand{\ran}{\mathop{\mathrm{ran}}}

\newcommand{\Pos}{\mathsf{PoSet}}
\newcommand{\msl}{\wedge\mathsf{SemiLatt}}
\newcommand{\filt}{{\Cal F}}

\newcommand{\Flt}[1]{\mathrm{Flt}(#1)}
\newcommand{\card}[1]{|#1|}
\newcommand{\X}{\Cal X}
\newcommand{\con}{\mathrm{con}\,}
\newcommand{\cf}[1]{\mathrm{cf}(#1)}
\newcommand{\pcell}[1]{\mathrm{cell}^*(#1)}

\theoremstyle{plain}
\newtheorem{thm}{Theorem}[section]
\newtheorem{lem}[thm]{Lemma}

\newtheorem*{lem*}{Lemma}
\newtheorem*{cor*}{Corollary}
\newtheorem*{thm*}{Theorem}
\newtheorem{prop}[thm]{Proposition}

\theoremstyle{definition}
\newtheorem{defn}{Definition}[section]
\theoremstyle{definition}
\newtheorem{notn}{Notation}[section]

\theoremstyle{remark}
\newtheorem{rem}{Remark}[section]
\let\origmaketitle\maketitle
\def\maketitle{
  \begingroup
  \def\uppercasenonmath##1{} 
  \let\MakeUppercase\relax 
  \origmaketitle
  \endgroup
}
\begin{document}

\title{Good ultrafilters and highly saturated models: a friendly explanation}
\author{Paul~E.~Lammert}
\email{pel1@psu.edu}
\address{Department of Physics, Pennsylvania State University, 
University Park, PA 16802}
\date{Dec. 6, 2015}
\begin{abstract}
Highly saturated models are a fundamental part of the model-theoretic
machinery of nonstandard analysis. Of the two methods for producing them,
ultrapowers constructed with the aid of $\kappa^+$-good ultrafilters
seems by far the less popular. Motivated by the hypothesis that this is
partly due to the standard exposition being somewhat dense, a presentation
is given which is designed to be easier to digest.
\end{abstract}
\maketitle

\section{Should you read this?}

The most difficult piece of the model-theoretic machinery underlying nonstandard analysis 
(NSA) is certainly the construction of {\it polysaturated} models
(more generally, $\kappa$-saturated for $\kappa > \omega^+$).
There are two main approaches:
the ultralimit method and the good ultrafilter method.
Stroyan and Luxemburg\cite{Stroyan+Luxemburg} and 
the recent volume\cite{Loeb+Wolff} edited by Loeb and Wolff contain detailed
proofs using ultralimits. This method has the advantage that at least the general
idea seems clear enough. Lindstr{\o}m's lectures\cite{Lindstrom85} is very
unusual in adopting the good ultrafilter method, which allows
use of plain ultrapower construction by relying on special properties of the ultrafilter. 
This facilitates using ultrapowers to directly reason about nonstandard universes;
the ultralimit construction is too unwieldy for that.
Lindstr{\o}m essentially copies the proof of Chang \& Keisler\cite{Chang+Keisler} 
(itself a fleshing-out of Kunen\cite{Kunen-72}) of the existence of 
$\kappa^+$-saturated ultraproducts,
describing it as ``an interesting, but quite technical piece of infinitary 
combinatorics''. To me, ``opaque'' seemed an apt description.
The following attempt at clarification is offered firstly
to others who are not model theorists or set theorists, but are 
interested in using nonstandard analysis (or model theory more generally) 
in their home fields. 
In mathematical essence, it follows \S 6.1 of Chang \& Keisler.
The difference is primarily one of presentation.

Here is a brief description of the body of the paper, highlighting novel elements.
\S \ref{sec:ultraproducts} and \ref{sec:saturation} very briefly recall 
ultraproducts and saturation.
Rather than just showing that the definition of {\it good ultrafilter} works,
\S \ref{sec:sufficient} follows some intution vaguely inspired by modal logic
to deduce a condition on an ultrafilter which endows an untraproduct with
$\kappa$-saturation.
The condition is reformulated and abstracted into an order-theoretic definition 
of {\it good} in \S \ref{sec:goodness} in \S \ref{sec:goodness} and used in the
first main Thm. \ref{thm:kappa+-saturated}.
That may seem to be gratuitous generalization, but it has some naturality, 
and the domain side of the abstract version is maintained through the 
second main Thm. \ref{thm:good-ultrafilter}, clarifying a little what is involved,
with no cost in complication.
In \S \ref{sec:construction}, the construction of a $\kappa^+$-good
ultrafilter on $\kappa$ is given, structured as a transfinite recursion with
some complications. We add sets step-by-step to a filter to produce
an ultrafilter, simultaneously building in goodness for a set of functions. 
(Which set is determined by the finished ultrafilter!) 
The recursion also requires a supply of a certain auxiliary raw 
material --- partitions of $\kappa$ with special properties, which are simply
consumed, not built into the final product in any meaningful sense. 
The construction of this supply, which is completely independent of everything else,
is segregated in \S \ref{sec:stock-of-partitions}.

\section{What's good about a good ultrafilter?}
\label{sec:what's-good}

\begin{notn}
Ordinals are denoted by greek letters $\alpha,\ldots,\mu$, but $\kappa$ will
always be an infinite cardinal and $\lambda$ and $\mu$ usually limit ordinals.
As is commonly done, we identify infinite cardinal numbers with initial ordinals.
To avert potential confusion, ordinal comparisons are denoted by curly symbols 
like `$\prec$', cardinal comparisons by `$<$', etc. 
The successor cardinal to $\kappa$ is $\kappa^+$.
$\Powfin(X)$ consists of the finite, and $\Powcof(X)$ of the cofinite, subsets of $X$.
$\Powtop(X)$ comprises the subsets with the same cardinality as $X$ itself 
({\it large} subsets).
The cardinality of a set $X$ is denoted $\card{X}$, and the underlying universe
of an ${\Cal L}$-structure ${\Cal B}$ by $|{\Cal B}|$.
\end{notn}
\subsection{Ultraproduct models\label{sec:ultraproducts}}

Let ${\Cal L}$ be a formal language, $({\Cal A}_i | {i\in I} )$ a family of
${\Cal L}$-structures indexed by the nonempty set $I$, and $\Cal U$ an ultrafilter on $I$.
These are the raw materials for the mod-${\Cal U}$ ultraproduct 
$\prod_{\Cal U} {\Cal A}_i$.
With $|{\Cal A}_i|$ the underlying universe of ${\Cal A}_i$, the universe 
of the ultraproduct is the quotient of $\prod_{i\in I} |{\Cal A}_i|$ 
by the equivalence relation $\simeq_{\Cal U}$, defined by,
$(a_i)_{i\in I} \simeq_{\Cal U} (b_i)_{i\in I}$ iff 
\hbox{$\setbld{i\in I}{a_i = b_i} \in {\Cal U}$.} 
The equivalence class of $a$ is denoted $[a]$.
The interpretation of ${\Cal L}$ in the ultraproduct is defined by:
$\prod_{\Cal U} {\Cal A}_i \models \phi$ iff
\hbox{$\setbld{i\in I}{{\Cal A}_i \models \phi} \in {\Cal U}$.}

\subsection{Saturation}
\label{sec:saturation}

A set of formulas $\Sigma(x)$ with one free variable is 
satisfiable in an ${\Cal L}$-structure ${\Cal B}$ if there is 
$b\in |{\Cal B}|$ such that ${\Cal B}\models \Sigma[b]$.
It is finitely satisfiable if every finite subset is satisfiable.
Given a subset $X \subset |{\Cal B}|$, the enrichment of the language
obtained by adding names for all the elements of $X$ is denoted ${\Cal L}(X)$, 
and the corresponding structure by $({\Cal B},a)_{a\in X}$ (`$a$' being interpreted by $a$!).
${\Cal B}$ is $\kappa$-saturated if for any $X\subset |{\Cal B}|$ with
$\card{X} < \kappa$, then a set of formulas $\Sigma(x) \subset {\Cal L}(X)$ with one 
free variable is satisfiable in $({\Cal B},a)_{a\in X}$ as soon as it is finitely satisfiable.

%
%
\begin{rem}
The Robinson-Zakon constructions in NSA invlve a bit more
than the generic ultraproducts appearing here:
an ultrapower is taken of the superstructure 
$\cup_{n\prec\omega} V_n(\X)$ over a set 
$\X$ [\hbox{$V_0(\X) \defeq \X$}, \hbox{$V_{n+1}(\X) \defeq V_{n}(\X)\cup\Pow(V_n(\X))$}];
elements of unbounded superstructure rank are ejected; and the result is embedded into 
the superstructure over $\Star{\X}$ via Mostowski collapse. 
The middle step results in a slightly different notion of saturation. 
Elements of bounded rank in the ultrapower are (identified, via Mostowski collapse, with) 
the {\it internal} entities.
In the previous description we should add requirements that $X$ consist of internal entities and
that for some $n$, $\Sigma(x)$ is finitely satisfiable in $V_n$.
The difference essentially boils down to just making sure that the ranks are bounded. 
Keeping that in mind, it is not difficult to specialize to the NSA setting.
\end{rem}

\subsection{Sufficient conditions on ${\Cal U}$ for saturation\label{sec:sufficient}}

Suppose the ultrafilter ${\Cal U}$ guarantees the condition 
(S): Whenever $\Sigma(x)$ with $|\Sigma| < \kappa$ is finitely satisfiable in 
a mod-${\Cal U}$ ultraproduct $\prod_{\Cal U} {\Cal A}_i$, then it is satisfiable.
Then, any mod-${\Cal U}$ ultraproduct interpreting a language with $\card{{\Cal L}} < \kappa$
is $\kappa$-saturated. This follows because for $X$ a subset of the ultraproduct with
$\card{X} < \kappa$, ${\Cal L}(X)$ is another language with cardinality less than
$\kappa$ and condition S immediately delivers what is required.
Thus, we are going to deduce a property of ${\Cal U}$ which guarantees condition S.

To that end, suppose $\Sigma(x)$ with $|\Sigma| < \kappa$ is finitely satisfiable
in $\prod_{\Cal U} {\Cal A}_i$. This can be phrased in terms of a satisfiability 
function
\begin{equation}
p:\Powfin(\Sigma) \rightarrow {\Cal U}  ::
\Theta \mapsto \setbld{i\in I}{ {\Cal A}_i \models \exists x (\wedge \Theta)}.
\end{equation}
$p(\Theta)$ is the set of indices of models in which the finite set of
formulas $\Theta$ is satisfiable.
A binary relation-style notation is also convenient: `$i\, p\, \Theta$'
is synonymous with `$i\in p(\Theta)$'.
Clearly, $p$ satisfies the condition
\begin{equation}
\Theta \subseteq \Theta^\prime \Rightarrow p(\Theta) \supseteq p(\Theta^\prime),
\;\text{ equivalently }\;
i \,p\, (\Theta \cup \Theta^\prime) \Rightarrow 
i \,p\, \Theta
\text{ \& } 
i \,p\, \Theta^\prime.
\tag{Psb}
\label{eq:possibility}
\end{equation}
defining the notion of a {\it possibility function} (or binary relation).

With $p$ denoting the possibility relation associated with $\Sigma$ as above,
$\Sigma$ is satisfiable if there exists $(\Phi_i \in \Powfin(\Sigma))_{i\in I}$ 
{\it actualizing} $p$ in the sense that
\begin{equation}
\text{(a): } i\,p\,\Phi_i,\text{ and }
\text{(b): } 
\text{for all } \theta\in\Sigma,\, 
\setbld{i\in I}{\theta \in \Phi_i} \in {\Cal U}.
\nonumber
\end{equation}
This is because, according to (a), $p$ certifies the existence of 
an element $[(a_i \,|\,i\in I)]$ of the ultraproduct
such that ${\Cal A}_i \models \Phi_i[a_i]$.
Clause (b) then shows that $[(a_i\,|\, i\in I)]$ satisfies every formula in $\Sigma$. 
As linguistic aspects of $\Sigma$ are untouched here, 
we might as well ask for criteria guaranteeing that
{every} possibility function on any set $X$, $\card{X} < \kappa$,
has an actualization.

Now, by basic properties of ultrafilters, clause (b) in the definition implies 
that \hbox{$\setbld{i\in I}{\Theta \subseteq \Phi_i} \in {\Cal U}$}
for every $\Theta\in\Powfin(X)$.
Thus, defining $p_\Phi$ by $i\, p_\Phi \, \Theta$ 
iff $\Theta \subseteq \Phi_i \in \Powfin(X)$, $p_\Phi$ is a
possibility relation in its own right.
It satisfies the possibility condition (\ref{eq:possibility}), and the 
strenghtened form of clause (b) shows that 
\hbox{$p_\Phi: \Powfin(X) \rightarrow {\Cal U}$}. 
Such a possibility relation $q=p_\Phi$ has some special properties:
\begin{equation}
\tag{Con}
\label{eq:consistency}
q \text { consistent: }\quad
i \,q\, (\Theta \cup \Theta^\prime) \Leftarrow
i \,q\, \Theta
\text{ \& } 
i \,q\, \Theta^\prime,
\end{equation}
which is inverse to the implication in (\ref{eq:possibility}), and 
\begin{equation}
q\text{ locally finite: } \quad
\text{ for all } i\in I,\;
\sup \setbld{\card{\Theta}}{i\,q\,\Theta} < \infty.
\tag{Lfn}
\label{eq:loc-fin}  
\end{equation}
(\ref{eq:consistency}) is equivalent to the statement that
$q^{-1} i = \setbld{\Theta \in \Powfin(X)}{i\, q \,\Theta}$
is closed under finite union.
Thus, if $q$ satisfies (\ref{eq:possibility},\ref{eq:consistency},\ref{eq:loc-fin}), 
then $\bigcup q^{-1} i$ is finite, hence itself a member of $q^{-1} i$,
for all $i\in I$. And, in that case $\bigcup q^{-1} i$ is an actualization of $q$.
So, a locally finite consistent possibility relation has an actualization.

Finally, suppose that $p^\prime \le p$ 
[meaning $p^\prime(\Theta) \subseteq p(\Theta)$ for $\Theta\in\Powfin(X)$].
Then, an actualization of $p^\prime$ is also an actualization of $p$ 
[only clause (a) sees $p$], and $p^\prime$ is locally finite if $p$ is.
Thus, if for every possibility relation $p$, there were two others,
$Tp \le p$ locally finite, and $Mp \le p$ consistent, 
then $MTp$ satisfies 
(\ref{eq:possibility},\ref{eq:consistency},\ref{eq:loc-fin}), hence
has an actualization which is also an actualization of $p$.
Summing up, ${\Cal U}$ guarantees saturation if the operators $M$ and $T$ exist.

$T$, at least, is quickly dispatched.
\begin{lem}
If the ultrafilter ${\Cal U}$ on $I$ is countably incomplete, then for every
possibility relation $p$, there is a locally finite possibility relation
$q\le p$.  
\end{lem}
\begin{proof}
That ${\Cal U}$ is countably incomplete means that there exists
$\{K_m\}_{m\in \Nat} \subseteq {\Cal U}$ 
such that $I_n \defeq \cap_{m\le n} K_m \downarrow \emptyset$.
Defining $N(i) = \min\{ n \in \Nat \,|\, i \notin I_n\}$ for $i\in I$,
note that 
\begin{equation}
\setbld{i\in I}{n \le N(i)} \in {\Cal U} \text{ for all } n\in\Nat.
\label{eq:N}
\end{equation}
Now, for $p:\Powfin(X) \rightarrow {\Cal U}$ a possibility function, 
define $Tp$ by
\begin{equation}
i \, (Tp)\, \Theta \text{ iff }
i \, p\, \Theta \text{ \& } \card{\Theta} \le N(i).
\nonumber
\end{equation}
(\ref{eq:N}) shows that $Tp$ is a possibility relation, and it is
manifestly locally finite.
\end{proof}

\subsection{Goodness}
 \label{sec:goodness}

If $X$ is an infinite set, then a map $\Powfin(X) \rightarrow Y$ can equally well
be viewed as a map $\Powcof(X) \rightarrow Y$, simply by precomposing with complementation.
The correspondingly rephrased condition for an $M$-operator for ${\Cal U}$
is as follows.
Given $f:\Powcof(X) \rightarrow {\Cal U}$
satisfying $x \subseteq y \Rightarrow f(x) \subseteq f(y)$, there is 
$(Mf:\Powcof(X) \rightarrow {\Cal U}) \le f$ satisfying $Mf(x\cap y) = Mf(x) \cap Mf(y)$.
Definition \ref{def:good} below captures the lattice-theoretic essence of this.
\begin{notn}
$\Pos$ denotes the category of posets with monotone maps, and
$\msl$, the subcategory of meet-semilattices with meet-preserving (or ``multiplicative'') 
maps. [A meet-semilattice is a poset with a binary operation $\wedge$ such that 
$a\le x$ \& $a\le y$ iff $a \le x\wedge y$.]
``$f: Y\rightarrow Z \; \text{ in }{\mathsf{C}}$'' indicates that $f$ is 
in the set $\mathsf{C}(Y,Z)$ of $\mathsf{C}$ morphisms. 
A subset of $\Pow(X)$ is a poset with $\subseteq$ in the role of $\leq$,
and if closed under intersection, is also a meet-semilattice with $\cap$ 
in the role of $\wedge$.
\end{notn}

\newcommand{\Y}{\Cal Y}
\begin{defn}[good]\label{def:good}
Let $\Y$ and $\X$ be meet-semilattices.
If, for \hbox{$f:\X \rightarrow \Y$} in $\Pos$, there is 
$g:\X \rightarrow \Y \text{ in }\msl$ with $g \le f$, then $\Y$ is
{\it $f$-good} (or ``good for $f$''). 
If $\Y$ is good for every $f$ in $\Pos(\X,\Y)$, 
then $\Y$ is {\it $\X$-good}.
Finally, for $\kappa$ an infinite cardinal, $\Y$ is
{\it $\kappa$-good} if it is $\Powcof(\beta)$-good for all $\beta < \kappa$.
\end{defn}

\begin{lem}\label{lem:inheritance}
Suppose $\Y$ is $\X$-good and
${\Cal W} \stackrel{s}{\rightarrow} \X \stackrel{r}{\rightarrow} {\Cal W}$ in $\msl$, with
\hbox{$r\circ s = id_{{\Cal W}}$}. Then, $\Y$ is ${\Cal W}$-good.
\end{lem}
\begin{proof}
Given $f:{\Cal W}\rightarrow \Y$ in $\Pos$, $\X$-goodness implies existence of
$g:\X\rightarrow \Y$ in $\msl$ with $g \le f\circ r$.
Then, $g\circ s \le f\circ r\circ s = f$ meets the requirement.
\end{proof}

\begin{thm}
\label{thm:kappa+-saturated}
If ${\Cal U}$ is a a $\Powcof(\kappa)$-good ultrafilter, then
any mod-${\Cal U}$ ultraproduct for a language of cardinality less than $\kappa^+$ 
is $\kappa^+$-saturated.
\end{thm}
\begin{proof}
Combine the preceding with the discussion in \S \ref{sec:sufficient}.
Lemma \ref{lem:inheritance} is not really needed, but
is applicable with ${\Cal W} = \Powcof(\beta)$, $\X = \Powcof(\kappa)$,
$s(x) = x \cup (\kappa\setminus\beta)$ and $r(y) = y\cap \beta$.
\end{proof}

Before taking up the main agenda, we pick off and enjoy some low-hanging fruit.

\begin{thm}
\label{thm:omega+good}
Any countably incomplete ultrafilter $\Cal U$ is $\omega^+$-good.
\end{thm}
\begin{proof}
For $f\in\Pos(\Powcof(\omega),{\Cal U})$,
$f{\mkern-6mu}\upharpoonright{\mkern-6mu}\X$ 
is multiplicative if $\X\subseteq\Powcof(\omega)$ is \hbox{totally} ordered,
since for $\Theta,\Theta^\prime \in \X$, 
$f(\min(\Theta,\Theta^\prime)) = f(\Theta\cap\Theta^\prime) \subseteq 
f(\Theta)\cap f(\Theta^\prime)$.
Thus, if
$M^*\in\msl(\Powcof(\omega),\X)$, then
$Mf \defeq f\circ M^* \in\msl(\Powcof(\omega),{\Cal U})$.
Further, if $M^* \le id$ ($M^*\Theta \subseteq \Theta$), then $Mf \le f$. 
These meet the requirements:
\begin{equation}
{\X}= \{X_n\}_{n\prec\omega} \text{ where } X_n \defeq\setbld{m\prec\omega}{n\preceq m},
\text{ and }
{M^*}\Theta \defeq \bigcup\setbld{x\in\X}{x\subseteq \Theta}.
\nonumber  
\end{equation}
\end{proof}


\section{how to make a $\kappa^+$-good ultrafilter\label{sec:construction}}

\begin{thm}
\label{thm:good-ultrafilter}
If $\X$ is a meet-semilattice with $\card{\X} \le \kappa$, then there
is an $\X$-good ultrafilter on $\kappa$.  
\end{thm}

\S \ref{sec:recursion} describes a transfinite recursion to construct the required ultrafilter.
Lemmata \ref{lem:making-good} and \ref{lem:deciding} show how to carry out
successor steps with the aid of a supply $\Pi$ of partitions of $\kappa$ so as to
preserve usability of these auxiliary resources.
The proof of Thm. \ref{thm:good-ultrafilter} is finished at the end of this Section,
but it depends on the existence of a suitable $\Pi$, the construction of
which is given \S \ref{sec:stock-of-partitions}.

\subsection{A transfinite recursion\label{sec:recursion}}

If $\X$ is a meet-semilattice with $\card{\X} = \kappa$, then 
$\card{\Pos(\X,\Pow(\kappa))} \le (2^\kappa)^\kappa = 2^{\kappa\cdot\kappa} = 2^\kappa$.
So, fix length-$2^\kappa$ enumerations 
\begin{equation}
 \Pow(\kappa) = \{J_\gamma\}_{\gamma \prec 2^\kappa}, \quad
\Pos(\X,\Pow(\kappa)) = \{f_\delta\}_{\delta \prec 2^\kappa},
\nonumber
\end{equation}
and, starting from a countably-incomplete filter $\filt_0$,
carry out a transfinite recursion of length $2^\kappa$ with the following
steps.
\begin{enumerate}
\item[{suc}$_\sigma$:]
At a successor stage $\sigma = \gamma + 1$, construct filter $\filt_{\gamma+1}$
satisfying
\begin{itemize}
\item  
$\filt_{\gamma+1} \supseteq \filt_{\gamma}$,
\item  
$\filt_{\gamma+1}$ decides $J_{\gamma + 1}$.
\item  
$\filt_{\gamma+1}$ is $f_\delta$-good, where $\delta$ is $\prec$-least in
\newline
$\setbld{\delta \prec 2^\kappa}{\ran f_\delta \subseteq \filt_\gamma 
\text{ but } \filt_\gamma \text{ not } f_\delta\text{-good}}$ 
(skip, if that set is empty).
\end{itemize}
\item[{lim}$_\lambda$:]
At a limit stage $\lambda$, take
$\filt_\lambda = \bigcup_{\gamma \prec \lambda} \filt_{\gamma}$.
\end{enumerate}

{If} step suc$_{\gamma+1}$ is possible for every $\gamma \prec 2^\kappa$, 
then \hbox{$\filt^* \defeq \cup\setbld{\filt_{\gamma}}{\gamma \prec 2^\kappa}$} is countably 
incomplete [it contains $\filt_0$], and is an ultrafilter 
[it decides every $J$ in $\Pow(\kappa)$]. 
$\kappa^+$-goodness is trickier, and involves the concept of {\it cofinality}. 
[For $\Y$ a poset, $X\subseteq \Y$ is cofinal in $\Y$ if, for all $y\in \Y$, there
exists $x\in X$, $y \preceq x$. The {\it cofinality} of $\Y$, denoted $\cf{\Y}$, 
is the least cardinal of sets cofinal in $\Y$.]
Suppose $f_\delta \in \Pos(\X,\Pow(\kappa))$ has range in $\filt^*$.
Since $\card{\ran{f_\delta}} \le \kappa < \cf{2^\kappa}$ 
[the second inequality by Prop. \ref{prop:strong-cantor} below], 
$\ran f_\delta \subset \filt_\gamma$ for some $\gamma < 2^\kappa$. 
The growing filter is made $f_\delta$-good
within another $\delta$ steps of the recursion, which point is certainly reached
because $2^\kappa$ is a cardinal, so that $\gamma + \delta \prec 2^\kappa$.

\begin{prop}
\label{prop:strong-cantor}
If $\alpha$ is an infinite cardinal and $2\le \lambda$, 
then $\alpha < \cf{\lambda^\alpha}$.  
\end{prop}
\begin{proof}
We prove this using the following \newline
Claim (K\"onig's theorem): if $\mu$ is an infinite cardinal, then 
$\mu < \mu^{\,\mathrm{cf}\,\mu}$.
\newline
Suppose the proposition is false, so
$\alpha \ge \mathrm{cf}\, \lambda^\alpha = \mathrm{cf}\, \mu$,
with abbreviation $\mu = \lambda^\alpha$.
Then, using the Claim in the first inequality that follows, 
\begin{equation}
\mu < \mu^{\mathrm{cf}\,\mu} \le \mu^\alpha = \lambda^{\alpha\cdot\alpha} = \lambda^\alpha = \mu.
\nonumber
\end{equation}
As this is a blatant contradiction, the Proposition is proven, once the Claim is secured.
Let $Y$ be an unbounded subset of $\mu$, and suppose that
an enumeration $\{g_\gamma\}_{\gamma < \mu}$ of all functions $Y \rightarrow \mu$ exists.
Diagonalize out of the enumeration as follows.
Define
$A_\delta \defeq \setbld{g_\gamma(\delta)}{\gamma < \mu}$ for $\delta\in Y$.
Since $\card{A_\delta} = \delta < \mu$, 
we can choose $f(\delta)$ from $\mu\setminus A_\delta \neq \emptyset$.
Then, $f \neq g_\gamma$ for every $\gamma$.
For, by unboundedness of $Y$, $\gamma < \delta \in Y$ for some $\delta$, implying that 
$f(\delta) \neq g_\gamma(\delta)$.
\end{proof}
\subsection{Partitions of $\kappa$ as a resource\label{sec:using-partitions}}

We begin with a simple order-theoretic observation. 
\begin{lem}
\label{lem:hat-map}
Suppose $\X$ is a meet-semilattice and $h:\X {\rightarrow} {\Cal Y}$, where
${\Cal Y} \subseteq \Pow(\kappa)$ is $\cup$-complete. Then,
the least $\Pos$ morphism $\hat{h}$ satisfying $h\le \hat{h}$ is
\begin{equation}
\label{eq:hat-map}  
\hat{h}: \X {\rightarrow} {\Cal Y} 
\,::\, x \mapsto \bigcup\setbld{h(y)}{y\le x}.
\end{equation}
If $h$ is disjoint, then $\hat{h}$ is also multiplicative.
\end{lem}
\begin{proof}
$\hat{h}$ is monotone, and clearly a monotone function as large as $h$ is as large as $\hat{h}$,
so the first assertion is immediate.
If $h$ is disjoint, given $\gamma \in \kappa$, there is {\em at most} one $w\in \X$ 
such that $\gamma\in h(w)$. So,
$\gamma \in \hat{h}(x\wedge y)$ iff (such $w$ exists and)
$w \le x\wedge y$, iff $w\le x$ and $w\le y$, 
iff $\gamma \in \hat{h}(x)$ and $\gamma \in \hat{h}(y)$.
Thus $\hat{h}$ is multiplicative.
\end{proof}
This result is immediately suggestive. Given $f:\X \rightarrow {\filt}$ 
in $\Pos$, we attempt to get a multiplicative $g \le f$ by cutting $f$ back
to a disjoint map $h \le f$ and then using the ``hat map''(\ref{eq:hat-map}).
Chopping up the image of $f$ with a partition is an obvious way of getting the 
disjoint $h$. That motivates the next couple of definitions.

\begin{defn}[partition, large, polycell, independent]\label{def:partition}
A {\it partition} $P$ is a set of disjoint subsets of $X$ ({\it cells} of $P$),
which cover $X$.
$P$ is {\it large} if $\card{P} = \card{X}$.
If $\Pi$ is a collection of partitions, a {\it polycell} is the intersection
$C_1\cap\cdots\cap C_n$ of a finite number of cells from distinct partitions
$P_1,\ldots,P_n \in \Pi$. The set of polycells of $\Pi$ is denoted $\pcell{\Pi}$.
$\Pi$ is {\it independent} (or, ``a set of independent partitions'') 
if $\emptyset \not\in \pcell{\Pi}$.
\end{defn}
\begin{defn}[consistent, polyconsistent]\label{def:consistent}
If $\filt$ is a filter on $Y$, then
\begin{equation}
\con \filt \defeq \setbld{y \in Y}{x\cap y \neq \emptyset, \forall x\in \filt}
\label{eq:consistent}
\end{equation}
is the set of subsets of $Y$ {\it consistent with} $\filt$.
A subset of (function into) $\Pow(Y)$ is called consistent with $\filt$ if 
it (its range) is a subset of $\con \filt$.
A collection $\Pi$ of partitions is {\it polyconsistent} with $\filt$ if
$\pcell{\Pi} \subseteq \con \filt$.
\end{defn}
\begin{rem}
\label{rem:consistency}
The notion of consistency is central, as Lemmata \ref{lem:making-good}, 
\ref{lem:deciding} will show.
A few points to note: 
$\con \filt$ is closed under taking supersets. 
$\filt\cup A$ generates a filter iff the $\cap$-semilattice generated by $A$ 
(finite intersections) is in $\con \filt$.
Finally, $Y\setminus(\con \filt)$ is closed under union.
For, if $a\cap x = b\cap y = \emptyset$ 
for $x,y\in \filt$, then 
$(a\cup b)\cap(x\cap y) = 
(a\cap x\cap y)\cup (b\cap x\cap y) = \emptyset$.
\end{rem}

\begin{lem}\label{lem:making-good}
Let $\Pi$ be a collection of large partitions of $\kappa$ polyconsistent with the
filter $\filt \subset \Pow(\kappa)$,
\begin{equation}
f: \X \rightarrow \filt\,\,{\text in}\,\, \Pos,
\nonumber  
\end{equation}
and $P \in \Pi$. 
Then there exists a filter 
$\filt^\prime = \Flt{\filt \cup \ran g} \supseteq \filt$ and
\begin{equation}
g:\X \rightarrow \filt^\prime\,\,{\text in}\,\msl,\quad g\le f
\nonumber  
\end{equation}
such that 
$\Pi^\prime = \Pi\setminus P$ is polyconsistent with $\filt^\prime$. 
\end{lem}
\begin{proof}
Since $\card{\X} \le \card{P} = \kappa$, there is an injective map
$D: \X \rightarrow P$. 
Since $\ran f \subseteq \filt$ and $P\subseteq \con \filt$, the map
\begin{equation}
h: \X \rightarrow \con \filt \,::\, s \mapsto f(s) \cap D(s)
\nonumber  
\end{equation}
is disjoint. 
Therefore, the range of $g \defeq \hat{h}$ is an $\cap$-semilattice in
$\con \filt$, and
$\filt^\prime \defeq \Flt{ \filt\cup \ran g} = \setbld{g(s)\cap A}{s\in \X, A\in \filt}$. 
For $C \in \pcell{\Pi\setminus P}$
and $g(s)\cap A\in\filt^\prime$,
$C \cap (g(s) \cap A)\supseteq
C \cap h(s) \cap A = [C \cap D(s)] \cap [f(s) \cap A]$.
This is nonempty because 
$C \cap D(s) \in \pcell{\Pi}$, $f(s) \cap A \in \filt$, and $\Pi$ is polyconsistent with $\filt$.
\end{proof}

\begin{lem}\label{lem:deciding}
Let $\Pi$ be a collection of large partitions of $\kappa$ polyconsistent with the
filter $\filt \subset \Pow(\kappa)$,
and $J \subseteq \kappa$.
Then, either
\begin{enumerate}
\item[(a)] $\Pi^\prime = \Pi$ is polyconsistent with 
$\filt^\prime = \Flt{ \filt\cup\{J\}  }$, or
\item[(b)] $\Pi^\prime \subseteq \Pi$ cofinite,
is polyconsistent with $\filt^\prime = \Flt{ \filt\cup\{\kappa\setminus J\}  }$.
\end{enumerate}
\end{lem}
\begin{proof}
Assume (a) fails.
Then, for a polycell $C$ of some finite 
\hbox{ $\{P_1,\ldots,P_n\} \subseteq \Pi$}, 
\hbox{$C \cap J \not\in \con \filt$.}
If (b) also fails with 
\hbox{$\Pi^\prime \defeq \Pi\setminus\{P_1,\ldots,P_n\}$},
there is similarly \hbox{$C^\prime\in \pcell{\Pi^\prime}$}
such that $C^\prime \cap (\kappa\setminus J)\not\in \con \filt$.
But then 
$C\cap C^\prime \subseteq (C\cap J) \cup (C^\prime \cap (\kappa\setminus J))$
is not in $\con \filt$ [see Rem. \ref{rem:consistency}], 
contradicting $\pcell{\Pi} \subseteq \con \filt$.
\end{proof}

\begin{proof}[Proof of Thm. \ref{thm:good-ultrafilter}]

Lemmata \ref{lem:making-good} and \ref{lem:deciding} show how the
successor steps of the recursion described in \S \ref{sec:recursion}
are to be carried out with the aid of some partitions of $\kappa$.
Given $\Pi_\gamma$ satisfying usable$_\gamma$: 
$\card{\Pi} = 2^\kappa$ and $\pcell{\Pi_\gamma} \subseteq \con \filt_\gamma$,
as well as $f:\X\rightarrow \filt_\gamma$ and $J\in\Pow(\kappa)$, the
transformations $(\filt,\Pi) \mapsto (\filt^\prime,\Pi^\prime)$ in
the Lemmata applied in sequence yield 
$\Pi_{\gamma+1}$ which is usable$_{\gamma+1}$, $f$-good, and decides $J$.

For limit stage lim$_\lambda$, $\lambda \prec 2^\kappa$, 
$\filt_\lambda = \cup_{\gamma \prec \lambda} \filt_\gamma$.
Add the rule $\Pi_{\lambda} = \cap_{\gamma \prec \lambda} \Pi_{\gamma}$, and
assume that ${\Pi_\gamma}$ is usable$_{\gamma}$ for all $\gamma \prec \lambda$.
Since $2^\kappa$ is a cardinal, $\card{\Pi_\lambda} = 2^\kappa$.
If $A \in \filt_\lambda$ and $C\in \pcell{\Pi_\lambda}$,
then $A\in \filt_\beta$ for some $\beta \prec \lambda$ and 
$C\in \pcell{\Pi_\beta}$, so $C\cap A \neq \emptyset$ by 
the inductive hypothesis. So, $\Pi_\lambda$ is polyconsistent with
$\filt_\lambda$ and thus usable$_{\lambda}$.

Finally, we need a countably incomplete filter $\filt_0$ 
and a usable$_0$ $\Pi_0$ to start the recursion.
For $n\prec\omega$, take 
${K}_n \defeq \setbld{ \alpha \prec \kappa }{\alpha = \lambda + m 
\text{ for a limit ordinal } \lambda, \text{ and } 0\le m \le n}$ 
($0$ is a limit ordinal).
Then, $K_n \downarrow \emptyset$, and
with $\mu$ the common cardinality of the $K_n$'s, $\mu\times\omega = \kappa$, 
so that $\mu = \kappa$.
Thus, $E=\{K_n\}_{n\prec\omega}$ satisfies the hypotheses of Prop. \ref{prop:Pi_0},
which delivers countably incomplete $\filt_0 = \Flt{E}$, as well as $\Pi_0$.
\end{proof}

\section{building a stock of large independent partitions of 
$\kappa$ \label{sec:stock-of-partitions}}

The goal of this section is Prop. \ref{prop:Pi_0}, which was used in the
proof of the main theorem \ref{thm:good-ultrafilter}.
A pair of clever lemmata are required.
\begin{lem}
\label{lem:collection-of-partitions}
There is a collection $\Pi^\prime$ of independent large partitions of $\kappa$ with
$\card{\Pi^\prime} = 2^\kappa$. 
\end{lem}
\begin{proof}
Define
  \begin{equation}
\Gamma = \{(s,r) \,|\, s\in \Powfin(\kappa), \,\, 
r:\, {s} \rightarrow \kappa \}.
\nonumber
  \end{equation}
Also, for each $J\subseteq \kappa$, the function
  \begin{equation}
f_J : \Gamma \rightarrow \kappa \,::\, (s,r) \mapsto r(s\cap J),
\nonumber
  \end{equation}
and corresponding partition of $\Gamma$
\begin{equation}
{P}_J^\Gamma = \setbld{f_J^{-1}(\gamma)}{\gamma\in\ran f_J}.
\nonumber
\end{equation}
We show that for distinct $J_1,\ldots,J_n\in \Pow(\kappa)$,
and arbitrary $\gamma_1,\ldots,\gamma_n\in \kappa$,
there is a solution $\xi$ of the system
\begin{equation}
f_{J_i}(\xi) = \gamma_i, \quad 1\le i \le n.
\label{eq:system-of-eqs}
\end{equation}
This immediately implies that 
$\setbld{P_J^\Gamma}{J\in \Pow(\kappa)}$ comprises
$2^\kappa$ [by (\ref{eq:system-of-eqs} with $n=2$], 
independent [general $n$],
large [$n=1$] partitions of $\Gamma$.

To solve (\ref{eq:system-of-eqs}), note that for each pair $1\le i< j\le n$, 
there is $x_{ij}\in J_i\Delta J_j$,
so, with $t \defeq \{x_{ij}\,|\, 1\le i < j \le n\}$,
the subsets $K_i \defeq t\cap J_i$ are all distinct, for $1\le i\le n$.
Then, $\xi = (t,r)$ does the trick if $r(K_i) = \gamma_i$. 

Since $\card{\Gamma} = {\kappa}$,
there is a bijection $\sigma: \kappa \rightarrow \Gamma$.
Use it to pull the $P_J^\Gamma$'s back to $\kappa$, obtaining 
partitions $P_J \defeq \sigma^{-1}P_J^\Gamma$.
\begin{equation}
\Pi^\prime \defeq \setbld{P_J}{J\in \Pow(\kappa)}
\label{eq:family-of-partitions}
\end{equation}
is the promised collection.
\end{proof}

\begin{lem}\label{lem:grand-hotel}
If $\beta \leq \kappa$ and 
$f:\beta \rightarrow \Powtop(\kappa)$, 
then there is a disjoint $h \le f$.
\end{lem}
\begin{proof}
By the axiom of Choice, there is a bijection between $\kappa$ and $\kappa\times\kappa$, 
inducing a (henceforth {\em the}) well-order on $\kappa\times\kappa$.
Let $I$ be the set of all {\em injective} functions from initial segments of 
$\kappa\times\kappa$ to $\kappa$ that map $x\times\kappa$ into $f(x)$,
ordered by extension ($g \le g^\prime$ iff $g^\prime$ agrees with $g$ on $\dom g$). 
We show that some $g^* \in I$ has $\dom g^* = \kappa\times\kappa$. For if not, there
is a least element $(x,\xi)$ not in the domain of any function in $I$, and a $g$ defined
on everything preceding $(x,\xi)$. But then $g(x\times\kappa)\subseteq f(x)$ has cardinality
{\em less than} $\kappa$, so an extension of $g$ can be defined at $(x,\xi)$. 
Contradiction.
The sets $h(x) \defeq g(x\times\kappa)\subseteq f(x)$ have 
cardinality $\kappa$ and are disjoint, as required. 
\end{proof}

\begin{prop}
\label{prop:Pi_0}
Let $E \subseteq \Powtop(\kappa)$ with $\card{E} \le \kappa$ 
have the finite intersection property.
Then, there is a set $\Pi$ of partitions of $\kappa$ with
$\card{\Pi} =2^\kappa$ and $\pcell{\Pi} \subseteq \con \Flt{E}$.
\end{prop}
\begin{proof}
$\hat{E} \defeq \setbld{\cap I}{I\in\Powfin(E)} \subseteq \Powtop(\kappa)$ 
is a filter base for $\filt = \Flt{E}$.
Since $\card{\hat{E}} \le \kappa$,
Lemma \ref{lem:grand-hotel} guarantees subsets
$I_e\subseteq e \in \hat{E}$ 
such that $\{I_e\}_{e\in\hat{E}}$ is disjoint, 
and $\card{I_e} = \kappa$ so that there are
bijections $\phi_e:I_e \rightarrow \kappa$. Using them, define a surjection
\begin{equation}
\phi:\kappa \rightarrow \kappa :: \,
x \mapsto 
\begin{cases}
\phi_e(x), \quad \text{ if } x\in I_e \text{ for some } e
\\
0, \quad \text{otherwise},
\end{cases}  
\nonumber
\end{equation}
bijective on each $I_e$.
Now, pull back $\Pi^\prime$ from (\ref{eq:family-of-partitions}) via $\phi$ to obtain
\begin{equation}
{\Pi} = \setbld{\hat{P}_J}{J\in\Pow(\kappa)},\quad 
  \hat{P}_J \defeq \phi^{-1} P_J,
\nonumber
\end{equation}
with $\card{{\Pi}} = \card{\Pi^\prime} = 2^\kappa$.
Since $\Pi^\prime$ is independent,
$\phi\!\upharpoonright\! I_e$ is bijective,
and pullback by $\phi$ commutes with intersection, 
${\Pi}\!\upharpoonright\! I_e$ is independent for each $e$.
Thus, since every member of $\filt$ is a superset of some $I_e$,
$\pcell{\Pi}\subseteq \con \filt$.
\end{proof}

\end{document}